\documentclass{amsart}

\usepackage[latin1]{inputenc}
\usepackage[english]{babel}
\usepackage{indentfirst}
\usepackage{amssymb}
\usepackage{amsthm}

\newcommand{\N}{\mathbb{N}}
\newcommand{\sub}{\subseteq}
\def\epsilon{\varepsilon}

\newtheorem{theo}{Theorem}[section]
\newtheorem{lem}[theo]{Lemma}

\numberwithin{equation}{section}

\title{On vector measures with values in $c_0(\kappa)$}

\author{Jos\'{e} Rodr\'{i}guez}

\address{Dpto. de Matem\'{a}ticas\\E.T.S. de Ingenieros Industriales de Albacete\\
Universidad de Castilla-La Mancha\\ 02071 Albacete\\ Spain} 

\email{jose.rodriguezruiz@uclm.es, joserr@um.es}



\thanks{The research was supported by grants PID2021-122126NB-C32 
(funded by MCIN/AEI/10.13039/501100011033 and ``ERDF A way of making Europe'', EU) and 
21955/PI/22 (funded by {\em Fundaci\'on S\'eneca - ACyT Regi\'{o}n de Murcia}).}

\begin{document}

\begin{abstract}
Let $\nu$ be a vector measure defined on a $\sigma$-algebra~$\Sigma$ and taking values in a Banach space. 
We prove that if $\nu$ is homogeneous and $L_1(\nu)$ is non-separable, 
then there is a vector measure $\tilde{\nu}:\Sigma \to c_0(\kappa)$ such that $L_1(\nu)=L_1(\tilde{\nu})$ 
with equal norms, where $\kappa$ is the density character of~$L_1(\nu)$. 
This is a non-separable version of a result of [G.P. Curbera, Pacific J. Math. 162 (1994), no.~2, 287--303].
\end{abstract}

\maketitle

\section{Introduction}

Spaces of integrable functions with respect to a vector measure play an important role in Banach lattices and operator theory. 
Every Banach lattice with order continuous norm and a weak unit is lattice-isometric to the $L_1$ space of some vector measure,
\cite[Theorem~8]{cur1} (cf. \cite[Proposition~2.4]{dep-alt}). Such a representation is not unique, in the sense that a Banach lattice
can be lattice-isometric to the $L_1$ spaces of completely different vector measures. 
The following result was proved in~\cite[Theorem~1]{cur2} (cf. \cite[Theorem~5]{lip} for a different proof): 

\begin{theo}[G. P. Curbera]\label{theo:Curbera}
Let $\nu$ be a vector measure defined on a $\sigma$-algebra~$\Sigma$ and taking values in a Banach space. 
If $\nu$ is atomless and $L_1(\nu)$ is separable, then there is a vector measure $\tilde{\nu}:\Sigma\to c_0$ such that 
$L_1(\nu)=L_1(\tilde{\nu})$ with equal norms.
\end{theo}

In general, this result is not valid for vector measures with atoms, as shown in~\cite[pp.~294--295]{cur2}.
It is natural to ask about non-separable versions of Theorem~\ref{theo:Curbera}
by using $c_0(\kappa)$ as target space for a large enough cardinal~$\kappa$. This question
was posed by Z.~Lipecki at the conference ``Integration, Vector Measures and Related Topics VI'' (Bed\l ewo, June 2014).
In~\cite{rod16} we provided some partial answers by using a certain superspace of~$c_0(\kappa)$, namely, the so-called
Pe\l czy\'{n}ski-Sudakov space. In the particular case $\kappa=\aleph_1$ (the first uncountable cardinal), this is the Banach space $\ell_\infty^c(\aleph_1)$ of all
bounded real-valued functions on~$\aleph_1$ with countable support.
   
In this note we refine the results of~\cite{rod16} by proving the following:

\begin{theo}\label{theo:general}
Let $\nu$ be a vector measure defined on a $\sigma$-algebra~$\Sigma$ and taking values in a Banach space. 
If $\nu$ is homogeneous and $L_1(\nu)$ is non-separable, then there is a vector measure $\tilde{\nu}:\Sigma\to c_0(\kappa)$ such that 
$L_1(\nu)=L_1(\tilde{\nu})$ with equal norms, where $\kappa$ is the density character of~$L_1(\nu)$.
\end{theo}

The proof of Theorem~\ref{theo:general} uses some ideas of \cite[Example~2.6]{rod16}. There we showed that, for an arbitrary
uncountable cardinal~$\kappa$ and $1<p<\infty$, the $L_p$ space of the usual product probability measure on the Cantor cube $\{-1,1\}^\kappa$ is equal to $L_1(\nu)$
for some $c_0(\kappa)$-valued vector measure~$\nu$. We stress that one cannot arrive at the same conclusion by using
a $c_0$-valued vector measure, see \cite[Example~4.16]{oka-rod-san}.

The paper is organized as follows. In Section~\ref{section:preliminaries}
we fix the terminology and include some preliminary facts on $L_1$ spaces of a vector measure.
In Section~\ref{section:proof} we prove Theorem~\ref{theo:general} and a similar result for non-homogeneous vector measures (Theorem~\ref{theo:omega1}). 
We finish the paper with further remarks on $c_0(\kappa)$-valued vector measures which might be of independent interest
(Theorem~\ref{theo:remarks}).

\section{Preliminaries}\label{section:preliminaries}

Our notation is standard as can be found in \cite{die-uhl-J} and~\cite{oka-alt}. We write $\mathbb N=\{1,2,\dots\}$. 
The {\em density character} of a topological space~$T$, denoted by~${\rm dens}(T)$,
is the minimal cardinality of a dense subset of~$T$. 

Given a non-empty set~$I$, we denote by $\Lambda_I$ 
the $\sigma$-algebra on~$\{-1,1\}^I$ generated by all the sets of the form 
$$
	\{x\in \{-1,1\}^I: \, x(i)=y(i) \text{ for all $i\in J$}\},
$$ 
where $J \sub I$ is finite and $y\in\{-1,1\}^J$. Every closed-and-open subset of~$\{-1,1\}^I$
is a finite union of sets as above. The symbol~$\lambda_I$ stands for the usual product probability measure
on~$(\{-1,1\}^I,\Lambda_I)$. For each $i\in I$ we denote by 
$$
	\pi_i^I:\{-1,1\}^I \to \{-1,1\}
$$ 
the $i$th-coordinate projection and, for each non-empty set $J \sub I$, we denote by 
$$
	\rho^I_{J}:\{-1,1\}^{I} \to \{-1,1\}^{J}
$$
the canonical projection.

All our Banach spaces are real.  
The closed unit ball of a Banach space~$X$ is denoted by~$B_X$ and the dual of~$X$ is denoted by~$X^*$. The symbol $\|\cdot\|_X$ stands for the norm of~$X$.
Given a non-empty set~$\Gamma$,
we denote by $c_0(\Gamma)$ the Banach space of all bounded functions $\varphi:\Gamma \to \mathbb R$
such that $\{\gamma\in \Gamma: |\varphi(\gamma)|>\epsilon\}$ is finite for every $\epsilon>0$, equipped with the supremum norm.

Let $(\Omega,\Sigma)$ be a measurable space. Given a Banach space~$X$, we denote by 
${\rm ca}(\Sigma,X)$ the set of all $X$-valued vector measures defined on~$\Sigma$. Unless stated otherwise, our {\em measures}
are meant to be countably additive. 

Let $\nu\in {\rm ca}(\Sigma,X)$. Given $A\in \Sigma$, we denote by $\nu_A$
the restriction of~$\nu$ to 
$$
	\Sigma_A:=\{B\in \Sigma: \, B \sub A\} 
$$
(which is a $\sigma$-algebra on~$A$). The set $A$ is called {\em $\nu$-null} if $\nu(B)=0$ for every $B\in \Sigma_A$ or, equivalently, 
$\|\nu\|(A)=0$, where $\|\nu\|$ is the semivariation of~$\nu$.
The family of all $\nu$-null sets is denoted by $\mathcal{N}(\nu)$. We say that $\nu$ is
{\em atomless} if for every $A\in \Sigma\setminus \mathcal{N}(\nu)$ there is $B\in \Sigma_A$ such that
neither $B$ nor $A\setminus B$ is $\nu$-null. 

By a {\em Rybakov control measure} of~$\nu$
we mean a finite non-negative measure of the form $\mu=|x^*\nu|$ for some $x^*\in X^*$ such that 
$\mathcal{N}(\mu) = \mathcal{N}(\nu)$ (see, e.g., \cite[p.~268, Theorem~2]{die-uhl-J});
here $|x^*\nu|$ is the variation of the signed measure $x^*\nu:\Sigma \to \mathbb R$ obtained as the composition of~$\nu$ with $x^*$. 

A $\Sigma$-measurable function $f:\Omega \to \mathbb{R}$ is said to be {\em $\nu$-integrable} if 
it is $|x^*\nu|$-integrable for all $x^*\in X^*$ and, for each $A\in \Sigma$, there is $\int_A f \, d\nu\in X$
such that
$$
	x^*\left ( \int_A f \, d\nu \right)=\int_A f\,d(x^*\nu)  	\quad\text{for all $x^*\in X^*$}.
$$
By identifying functions which coincide $\nu$-a.e., the set $L_1(\nu)$ of all (equivalence classes of) $\nu$-integrable functions is a Banach lattice with the $\nu$-a.e. order and the norm
$$
	\|f\|_{L_1(\nu)}:=\sup_{x^*\in B_{X^*}}\int_\Omega |f|\,d|x^*\nu|.
$$
We write ${\rm sim}(\Sigma)$ to denote the linear subspace of $L_1(\nu)$ consisting of all (equivalence classes of) {\em simple functions}, that is,
linear combinations of characteristic functions~$\chi_A$ where $A\in \Sigma$. The set ${\rm sim}(\Sigma)$ is norm dense in~$L_1(\nu)$. 
As in the case of finite non-negative measures, if $L_1(\nu)$ is infinite-dimensional, then its density character coincides with the 
minimal cardinality of a set $\mathcal{C} \sub \Sigma$ satisfying that $\inf_{C\in \mathcal{C}}\|\nu\|(A\triangle C)=0$ for all $A\in \Sigma$. 
We say that $\nu$ is {\em homogeneous} if it is atomless
and 
$$
	{\rm dens}(L_1(\nu))={\rm dens}(L_1(\nu_A))
	\quad\text{for every $A\in \Sigma\setminus \mathcal{N}(\nu)$}.
$$ 
In this case, 
the cardinal ${\rm dens}(L_1(\nu))$ is called the {\em Maharam type} of~$\nu$.
It is easy to check that: (i)~${\rm dens}(L_1(\nu))={\rm dens}(L_1(\mu))$ for any Rybakov control measure~$\mu$ of~$\nu$;
(ii)~$\nu$ is atomless (resp., homogeneous) if and only if some/any Rybakov control measure of~$\nu$ is atomless (resp., homogeneous).

As a Banach lattice, $L_1(\nu)$ has order continuous norm
and a weak unit (the function $\chi_\Omega$). If $\mu$ is a Rybakov control measure of~$\nu$, then $L_1(\nu)$ is a 
{\em K\"{o}the function space} over $(\Omega,\Sigma,\mu)$ and we can consider its {\em K\"{o}the dual}
$$
	L_1(\nu)':=\{g\in L_1(\mu): \, fg \in L_1(\mu) \text{ for all $f\in L_1(\nu)$}\}.
$$
For each $g \in L_1(\nu)'$ we have a functional $\varphi_g\in L_1(\nu)^*$ defined by 
$$
	\varphi_g(f):=\int_\Omega fg \, d\mu 
	\quad\text{for all $f\in L_1(\nu)$}. 
$$
Since $L_1(\nu)$ has order continuous norm, the equality 
\begin{equation}\label{eqn:Kothe}
	L_1(\nu)^*=\{\varphi_g: \, g\in L_1(\nu)'\} 
\end{equation}
holds
(see, e.g., \cite[p.~29]{lin-tza-2}).

We will also need the following two auxiliary results, which can be found in \cite[Lemma~2.3]{rod16} and 
\cite[Lemma~3.6]{nyg-rod}, respectively.

\begin{lem}\label{lem:NormLINFTYvalued}
Let $\Gamma$ be a non-empty set and let $Z$ be a closed subspace of~$\ell_\infty(\Gamma)$. For each $\gamma\in \Gamma$, denote by 
$e_\gamma^*\in B_{\ell_\infty(\Gamma)^*}$ the $\gamma$-th coordinate projection. 
Let $(\Omega,\Sigma)$ be a measurable space and let $\nu\in ca(\Sigma,Z)$. Then
$$
	\|f\|_{L_1(\nu)}=\sup_{\gamma\in \Gamma}\int_\Omega|f| \, d|e_\gamma^*\nu|
	\quad\text{for every $f\in L_1(\nu)$}.
$$
\end{lem}

\begin{lem}\label{lem:Lipecki}
Let $X$ and $Y$ be Banach spaces, let $(\Omega,\Sigma)$ be a measurable space and let $\nu\in {\rm ca}(\Sigma,X)$ and $\tilde{\nu}\in {\rm ca}(\Sigma,Y)$
such that $\mathcal{N}(\nu)=\mathcal{N}(\tilde{\nu})$. Suppose that there is a constant $c>0$ such that 
$\|f\|_{L_1(\nu)} \leq c \|f\|_{L_1(\tilde{\nu})}$ for every $f\in {\rm sim}(\Sigma)$.
Then $L_1(\tilde{\nu})$ embeds continuously into $L_1(\nu)$ with norm $\leq c$.
\end{lem}

\section{Results}\label{section:proof}

Given a probability space $(\Omega,\Sigma,\mu)$, we consider the equivalence relation on~$\Sigma$ defined by $A \sim B$ if and only if $\mu(A\triangle B)=0$.
The set $\Sigma/\mathcal{N(\mu)}$ of equivalence classes becomes a {\em measure algebra}
with the usual Boolean algebra operations and the functional defined by $\mu^{\bullet}(A^{\bullet}):=\mu(A)$ for all $A\in \Sigma$, where $A^{\bullet}\in \Sigma/\mathcal{N}(\mu)$ is the equivalence class of~$A$. We refer to~\cite{fre14} for more information on measure algebras.

We now proceed with the proof of our main result.

\begin{proof}[Proof of Theorem~\ref{theo:general}]
We divide the proof into several steps.

{\sc Step~1.}
Let $(\Omega,\Sigma)$ be the underlying measurable space and
let $\mu$ be a Rybakov control measure of~$\nu$. Suppose without loss of generality that $\mu(\Omega)=1$. 
Since $\mu$ is homogeneous and has Maharam type~$\kappa$, Maharam's theorem (see, e.g., \cite[Section~3]{fre14} or \cite[\S14]{lac-J}) ensures
that the measure algebras of~$\mu$ and~$\lambda_{\kappa}$ are 
isomorphic, that is, there is a Boolean algebra isomorphism
$$
	\theta: \Sigma/\mathcal{N}(\mu) \to \Lambda_{\kappa}/\mathcal{N}(\lambda_{\kappa})
$$
such that $\lambda_{\kappa}^{\bullet}\circ \theta=\mu^{\bullet}$. This isomorphism induces a lattice isometry 
$$
	\Phi:L_1(\mu)\to L_1(\lambda_{\kappa})
$$
such that for every $f\in L_1(\mu)$ we have
\begin{equation}\label{eqn:i}
	\int_\Omega f\, d\mu=\int_{\{-1,1\}^\kappa}\Phi(f)\, d\lambda_{\kappa}
\end{equation}
and
\begin{equation}\label{eqn:ii}
 	\Phi(f\chi_A)=\Phi(f)\chi_{C} \quad\text{whenever $A\in \Sigma$ and $C\in \Lambda_{\kappa}$ satisfy $\theta(A^{\bullet})=C^{\bullet}$.}
\end{equation}

{\sc Step~2.} It is well known that for an arbitrary Banach space~$Y$ the inequalities
$$
	{\rm dens}(Y^*,\text{weak}^*) \leq
	{\rm dens}(B_{Y^*},\text{weak}^*) \leq
	{\rm dens}(Y)
$$
hold. If, in addition, $Y$ is weakly compactly generated, then they turn out to be equalities
(see, e.g., \cite[Theorem~13.3]{fab-ultimo}). Therefore, since $L_1(\nu)$ is weakly compactly generated, \cite[Theorem~2]{cur1} (cf. \cite[p.~193]{buk-alt}), we have 
$$
	{\rm dens}(B_{L_1(\nu)^*},\text{weak}^*)={\rm dens}(L_1(\nu))=\kappa.
$$
Let $H \sub B_{L_1(\nu)^*}$ be a weak$^*$-dense subset of~$B_{L_1(\nu)^*}$
with cardinality~$\kappa$. Let us write $H=\{\varphi_{h_\alpha}:\alpha<\kappa\}$ where $h_\alpha\in L_1(\nu)'$ for all $\alpha<\kappa$
(see equality~\eqref{eqn:Kothe} at page~\pageref{eqn:Kothe}).
Then
$$
	\|f\|_{L_1(\nu)}=\sup_{\alpha<\kappa}\varphi_{h_\alpha}(f)
	\quad\text{for all $f\in L_1(\nu)$}.
$$
Since $|h_\alpha|\in L_1(\nu)'$ and $\varphi_{|h_\alpha|} \in B_{L_1(\nu)^*}$ for all $\alpha<\kappa$, the previous equality yields
\begin{equation}\label{eqn:norming0}
		\|f\|_{L_1(\nu)}=\sup_{\alpha<\kappa}\varphi_{|h_\alpha|}(|f|)
	\quad\text{for all $f\in L_1(\nu)$}.
\end{equation}

Fix $\alpha<\kappa$. Then $h_\alpha\in L_1(\mu)$ and so we can consider $\Phi(h_\alpha)\in L_1(\lambda_\kappa)$.
Hence, there exist a countable set $I_\alpha \sub \kappa$ and $\tilde{h}_\alpha\in L_1(\lambda_{I_\alpha})$
such that 
\begin{equation}\label{eqref:countablymany}
	\Phi(h_\alpha)=\tilde{h}_\alpha \circ \rho^\kappa_{I_\alpha}
\end{equation}
(see, e.g., \cite[254Q]{freMT-2}). 

Let $\psi: \kappa \to \kappa$ be an injective map such that 
$\psi(\alpha)\not\in I_\alpha$ for all $\alpha<\kappa$. Note that such a map can be constructed by transfinite induction. Indeed, take
$\alpha<\kappa$ and suppose that $\psi(\beta)$ is already defined for all $\beta<\alpha$. Then 
$\{\psi(\beta):\beta<\alpha\} \cup I_\alpha$ has cardinality strictly less
than~$\kappa$ (bear in mind that $\kappa$ is uncountable and $I_\alpha$ is countable). Hence, we can pick 
$\psi(\alpha)\in \kappa \setminus \{\psi(\beta):\beta<\alpha\} \cup I_\alpha$.

{\sc Step~3.} Fix $\alpha<\kappa$. We write
$$
	\pi^\kappa_{\psi(\alpha)}=\chi_{C_{\psi(\alpha)}}-\chi_{\{-1,1\}^\kappa \setminus C_{\psi(\alpha)}},
$$ 
where $C_{\psi(\alpha)}:=(\pi^\kappa_{\psi(\alpha)})^{-1}(\{1\}) \in \Lambda_\kappa$. Then
$$
	\Phi^{-1}(\pi^\kappa_{\psi(\alpha)})=\chi_{A_{\psi(\alpha)}}-\chi_{\Omega\setminus A_{\psi(\alpha)}},
$$ 
where $A_{\psi(\alpha)}$ is some element of~$\Sigma$ with $\theta(A_{\psi(\alpha)}^{\bullet})=C_{\psi(\alpha)}^{\bullet}$.
Since $|\Phi^{-1}(\pi^\kappa_{\psi(\alpha)})|= \chi_\Omega$ and $\varphi_{h_\alpha} \in B_{L_1(\nu)^*}$, we have 
$$
	g_\alpha:= h_\alpha \Phi^{-1}(\pi^\kappa_{\psi(\alpha)})\in L_1(\nu)'
$$
with $\varphi_{g_\alpha}\in B_{L_1(\nu)^*}$ and so 
$$
	\Big|\int_A g_\alpha \, d\mu\Big| = |\varphi_{g_\alpha}(\chi_A)| \leq
	\|\chi_A\|_{L_1(\nu)}\|\varphi_{g_\alpha}\|_{L_1(\nu)^*}\leq
	\|\chi_A\|_{L_1(\nu)}=\|\nu\|(A)
$$
for every $A\in \Sigma$. Hence, we have
$$
	\tilde{\nu}(A):=\left(\int_A g_\alpha \, d\mu\right)_{\alpha<\kappa}\in \ell_\infty(\kappa)
$$ 
and 
\begin{equation}\label{eqn:ca}
	\|\tilde{\nu}(A)\|_{\ell_\infty(\kappa)}\leq \|\nu\|(A)
\end{equation}
for every $A\in \Sigma$. Clearly, $\tilde{\nu}: \Sigma \to \ell_\infty(\kappa)$ is finitely additive. Since 
$\|\nu\|(A)\to 0$ as $\mu(A)\to 0$ (see, e.g., \cite[p.~10, Theorem~1]{die-uhl-J}), inequality~\eqref{eqn:ca} ensures that
$\tilde{\nu}$ is countably additive, that is, $\tilde{\nu}\in{\rm ca}(\Sigma,\ell_\infty(\kappa))$. 

{\sc Step~4.} Let $C \sub \{-1,1\}^\kappa$ be an arbitrary closed-and-open set (in particular, $C\in \Lambda_\kappa$) and let $A\in \Sigma$ such that $\theta(A^{\bullet})=C^{\bullet}$.
We claim that for every sequence $(\alpha_n)_{n\in \N}$ of pairwise distinct elements
of~$\kappa$ we have $\int_A g_{\alpha_n} \, d\mu = 0$ for $n$ large enough. 

Indeed, since $C$ is closed-and-open, there exist a finite set $I \sub \kappa$ and a set $B \sub \{-1,1\}^I$ such that
$C=B\times \{-1,1\}^{\kappa\setminus I}$. 
For each $n\in \N$ we define $J_n:=I\cup I_{\alpha_n}$ and, bearing in mind~\eqref{eqref:countablymany}, we write 
\begin{equation}\label{eqn:countablymany2}
	\Phi(h_{\alpha_n})\chi_C=\hat{h}_{\alpha_n} \circ \rho^\kappa_{J_n}
\end{equation}
for some $\hat{h}_{\alpha_n} \in L_1(\lambda_{J_n})$. Since $I$ is finite,
$\psi$ is injective and the $\alpha_n$'s are pairwise distinct,
there is $n_0\in \N$ such that for every $n\geq n_0$ we have $\psi(\alpha_n)\not\in J_n$, thus 
\begin{equation}\label{eqn:independence}
	\int_{\{-1,1\}^{\kappa \setminus J_n}} \pi^{\kappa\setminus J_n}_{\psi(\alpha_n)} \, 
	d\lambda_{\kappa\setminus J_n}=0
\end{equation}
and so
\begin{eqnarray*}
	\int_A g_{\alpha_n} \, d\mu & =&
	\int_\Omega h_{\alpha_n}(\chi_{A_{\psi(\alpha_n)}\cap A}-\chi_{A\setminus A_{\psi(\alpha_n)}}) \, d\mu
	\\ &\stackrel{\text{\eqref{eqn:i} \& \eqref{eqn:ii}}}{=}&
	\int_{\{-1,1\}^\kappa} \Phi(h_{\alpha_n})(\chi_{C_{\psi(\alpha_n)}\cap C}-\chi_{C\setminus C_{\psi(\alpha_n)}}) \, d\lambda_\kappa \\ &=&
	\int_{\{-1,1\}^\kappa} \Phi(h_{\alpha_n})\chi_{C} \, \pi^\kappa_{\psi(\alpha_n)} \, d\lambda_\kappa \\ & \stackrel{\eqref{eqn:countablymany2}}{=}&
	\int_{\{-1,1\}^\kappa} \big(\hat{h}_{\alpha_n} \circ \rho^\kappa_{J_n}\big) \, 
	\big(\pi^{\kappa\setminus J_n}_{\psi(\alpha_n)}\circ \rho^\kappa_{\kappa\setminus J_n}\big) \, d\lambda_\kappa \\ &\stackrel{\text{(*)}}{=}&
	\left(\int_{\{-1,1\}^{J_n}} \hat{h}_{\alpha_n} \, d\lambda_{J_n}\right) \left(\int_{\{-1,1\}^{\kappa \setminus J_n}} \pi^{\kappa\setminus J_n}_{\psi(\alpha_n)} \, 
	d\lambda_{\kappa\setminus J_n}\right)\stackrel{\eqref{eqn:independence}}{=}0,
\end{eqnarray*}
where equality (*) follows from Fubini's theorem.

{\sc Step~5.} We claim that 
$$
	\tilde{\nu}(A)\in c_0(\kappa) \quad\text{for every $A\in \Sigma$}
$$
and so $\tilde{\nu}\in {\rm ca}(\Sigma,c_0(\kappa))$.

Indeed, fix $A\in \Sigma$ and $\epsilon>0$. 
Choose $\delta>0$ such that 
$$
	\|\tilde{\nu}(B)\|_{\ell_\infty(\kappa)}\leq \frac{\epsilon}{2}
	\quad \text{for every $B\in \Sigma$ with $\mu(B)\leq \delta$}
$$
(see {\sc Step~3}). Take $C\in \Lambda_\kappa$ such that $\theta(A^{\bullet})=C^{\bullet}$. 
There is a closed-and-open set $C_\epsilon \sub \{-1,1\}^\kappa$ such that $\lambda_\kappa(C\triangle C_\epsilon)\leq \delta$. 
Take $A_\epsilon\in \Sigma$ such that $\theta(A_\epsilon^{\bullet})=C_\epsilon^{\bullet}$. 
By {\sc Step~4}, we have $\tilde{\nu}(A_\epsilon)\in c_0(\kappa)$. Since $\mu(A\triangle A_\epsilon)=\lambda_\kappa(C\triangle C_\epsilon)\leq \delta$, we have
\begin{eqnarray*}
	\|\tilde{\nu}(A)-\tilde{\nu}(A_\epsilon)\|_{\ell_\infty(\kappa)}
	&=&\|\tilde{\nu}(A\setminus A_\epsilon)-\tilde{\nu}(A_\epsilon\setminus A)\|_{\ell_\infty(\kappa)}\\ &\leq& 
	\|\tilde{\nu}(A\setminus A_\epsilon)\|_{\ell_\infty(\kappa)}+\|\tilde{\nu}(A_\epsilon\setminus A)\|_{\ell_\infty(\kappa)}
	\leq \epsilon. 
\end{eqnarray*}
As $\epsilon>0$ is arbitrary, $\tilde{\nu}(A_\epsilon)\in c_0(\kappa)$ and $c_0(\kappa)$ is a closed subspace of~$\ell_\infty(\kappa)$, 
it follows that $\tilde{\nu}(A) \in c_0(\kappa)$. This proves the claim.
  
{\sc Step~6.} Fix $f\in {\rm sim}(\Sigma)$. By Lemma~\ref{lem:NormLINFTYvalued} and the very definition of~$\tilde{\nu}$, we have 
$$
	\|f\|_{L_1(\tilde{\nu})}=\sup_{\alpha<\kappa}\int_{\Omega} |f g_\alpha| \, d\mu  =
	\sup_{\alpha<\kappa}\int_{\Omega} |f h_\alpha| \, d\mu=
	\sup_{\alpha<\kappa}\varphi_{|h_\alpha|}(|f|)\stackrel{\eqref{eqn:norming0}}{=}\|f\|_{L_1(\nu)}.
$$
In particular, $\mathcal{N}(\nu)=\mathcal{N}(\tilde{\nu})$ and we can apply 
Lemma~\ref{lem:Lipecki} twice to infer that $L_1(\nu)=L_1(\tilde{\nu})$ with equal norms. The proof is finished.
\end{proof}

The following lemma will be useful when dealing with non-homogeneous vector measures. Let us recall first a standard renorming
for the $L_1$ space of a vector measure. Let $X$ be a Banach space, let $(\Omega,\Sigma)$ be a measurable space
and let $\nu\in {\rm ca}(\Sigma,X)$. Then the formula
$$
	|||f|||_{L_1(\nu)}:=\sup_{A\in \Sigma}\left\|\int_A f \, d\nu\right\|_X, \quad
	f\in L_1(\nu),
$$
defines an equivalent norm on~$L_1(\nu)$ and, in fact, one has
\begin{equation}\label{eqn:norms}
	|||f|||_{L_1(\nu)}\leq \|f\|_{L_1(\nu)} \leq 2|||f|||_{L_1(\nu)}
	\quad
	\text{for all $f\in L_1(\nu)$}
\end{equation}
(see, e.g., \cite[p.~112]{oka-alt}).

\begin{lem}\label{lem:glue}
Let $X$, $X_1$ and $X_2$ be Banach spaces, let $(\Omega,\Sigma)$ be a measurable space and let $\nu\in {\rm ca}(\Sigma,X)$. 
Let $A_1,A_2 \in \Sigma$ be disjoint with $\Omega=A_1\cup A_2$ and let $\nu_i\in {\rm ca}(\Sigma_{A_i},X_i)$ such that
$L_1(\nu_{A_i})=L_1(\nu_i)$ with equivalent norms for $i\in \{1,2\}$. Define $\tilde{\nu}:\Sigma \to X_1\oplus_\infty X_2$ by
$$
	\tilde{\nu}(A):=(\nu_1(A\cap A_1),\nu_2(A\cap A_2))
	\quad\text{for all $A\in \Sigma$}.
$$ 
Then $\tilde{\nu}\in {\rm ca}(\Sigma,X_1\oplus_\infty X_2)$ and $L_1(\nu)=L_1(\tilde{\nu})$ with equivalent norms.
\end{lem}
\begin{proof}
Write $Z:=X_1\oplus_\infty X_2$. Clearly, $\tilde{\nu}\in {\rm ca}(\Sigma,Z)$. Let $c$ and $d$ be positive constants
such that
\begin{equation}\label{eqn:eq1}
	c^{-1} \|f|_{A_1}\|_{L_1(\nu_1)} \leq \|f|_{A_1}\|_{L_1(\nu_{A_1})} 
	\leq c \|f|_{A_1}\|_{L_1(\nu_1)}
\end{equation}
and
\begin{equation}\label{eqn:eq2}
	d^{-1} \|f|_{A_2}\|_{L_1(\nu_2)} \leq \|f|_{A_2}\|_{L_1(\nu_{A_2})} 
	\leq d \|f|_{A_2}\|_{L_1(\nu_2)}
\end{equation}
for every $f\in {\rm sim}(\Sigma)$. 

On the one hand, we have
\begin{equation}\label{eqn:eq3}
	\|f\|_{L_1(\nu)} \leq 2(c+d)\|f\|_{L_1(\tilde{\nu})}
\end{equation}
for every $f\in {\rm sim}(\Sigma)$. Indeed, note that 
\begin{equation}\label{eqn:integral}
	\int_A f \, d\tilde{\nu}=
	\left(\int_{A\cap A_1}f|_{A_1} \, d\nu_1, \int_{A\cap A_2}f|_{A_2} \, d\nu_2
	\right) \quad
	\text{for all $A\in \Sigma$}
\end{equation}
and so for each $i\in \{1,2\}$ we have
\begin{equation}\label{eqn:int}
	\|f|_{A_i}\|_{L_1(\nu_{i})} \stackrel{\eqref{eqn:norms}}{\leq} 2\sup_{A\in \Sigma}\left\|\int_{A\cap A_i}f|_{A_i} \, d\nu_i\right\|_{X_i}
	\stackrel{\eqref{eqn:integral}}{\leq} 2\sup_{A\in \Sigma}\left\|\int_{A}f \, d\tilde{\nu}\right\|_{Z} \stackrel{\eqref{eqn:norms}}{\leq} 2 \|f\|_{L_1(\tilde{\nu})}.
\end{equation}
It follows that
\begin{eqnarray*}
	\|f\|_{L_1(\nu)} & \leq & \|f\chi_{A_1}\|_{L_1(\nu)}+\|f\chi_{A_2}\|_{L_1(\nu)} \\ &=&
	\|f|_{A_1}\|_{L_1(\nu_{A_1})}+\|f|_{A_2}\|_{L_1(\nu_{A_2})}\\
	& \stackrel{\text{\eqref{eqn:eq1} \& \eqref{eqn:eq2}}}{\leq} &
	c\|f|_{A_1}\|_{L_1(\nu_{1})}+d\|f|_{A_2}\|_{L_1(\nu_{2})} \stackrel{\text{\eqref{eqn:int}}}{\leq} 2(c+d)\|f\|_{L_1(\tilde{\nu})}.
\end{eqnarray*}
This proves inequality~\eqref{eqn:eq3}.

On the other hand, we have
\begin{equation}\label{eqn:eq4}
	\|f\|_{L_1(\tilde{\nu})} \leq
	2\max\{c,d\} \|f\|_{L_1(\nu)}
\end{equation}
for every $f\in {\rm sim}(\Sigma)$. Indeed, observe that
\begin{eqnarray*}
	\|f\|_{L_1(\tilde{\nu})}&\stackrel{\eqref{eqn:norms}}{\leq}&
	2\sup_{A\in \Sigma}\left\|\int_A f \, d\tilde{\nu}\right\|_{Z}\\
	&\stackrel{\text{\eqref{eqn:integral}}}{=}&2 \sup_{A\in \Sigma}\max\left\{\left\|\int_{A\cap A_1}f|_{A_1} \, d\nu_1\right\|_{X_1}, \left\|
	\int_{A\cap A_2}f|_{A_2} \, d\nu_2\right\|_{X_2}\right\} \\
	& \stackrel{\eqref{eqn:norms}}{\leq} &
	2 \max\{\|f|_{A_1}\|_{L_1(\nu_1)},\|f|_{A_2}\|_{L_1(\nu_2)}\} \\
	& \stackrel{\text{\eqref{eqn:eq1} \& \eqref{eqn:eq2}}}{\leq} &
	2 \max\{c\|f|_{A_1}\|_{L_1(\nu_{A_1})},d\|f|_{A_2}\|_{L_1(\nu_{A_2})}\} \\
	&\leq & 2 \max\{c,d\}\|f\|_{L_1(\nu)},
\end{eqnarray*}
as claimed. 

Finally, inequalities~\eqref{eqn:eq3} and~\eqref{eqn:eq4} allow to apply Lemma~\ref{lem:Lipecki} to deduce that
$L_1(\nu)=L_1(\tilde{\nu})$ with equivalent norms.
\end{proof}

\begin{theo}\label{theo:omega1}
Let $X$ be a Banach space, let $(\Omega,\Sigma)$ be a measurable space and let $\nu\in {\rm ca}(\Sigma,X)$. If $\nu$ is atomless
and $L_1(\nu)$ has density character~$\aleph_k$ for some $k\in \N$, then there is $\tilde{\nu}\in {\rm ca}(\Sigma,c_0(\aleph_k))$ such that 
$L_1(\nu)=L_1(\tilde{\nu})$ with equivalent norms.
\end{theo}
\begin{proof}
Let $\mu$ be a Rybakov control measure of~$\nu$. Then $\mu$ is atomless
and $L_1(\mu)$ has density character~$\aleph_k$. Therefore, there exists a finite partition $\{A_1,\dots,A_n\}$ of~$\Omega$ consisting of elements of~$\Sigma$
such that, for each $i\in \{1,\dots,n\}$, the restriction of~$\mu$ to~$\Sigma_{A_i}$ is homogeneous and has Maharam type~$\aleph_{m_i}$ for some $m_i \in \N\cup\{0\}$
satisfying $m_1<m_2<\ldots<m_n=k$ (see, e.g., \cite[Section~3]{fre14} or \cite[p.~122, Theorem~7]{lac-J}). Now, for each $i\in \{1,\dots,n\}$
we can apply either Theorem~\ref{theo:Curbera} or Theorem~\ref{theo:general} to~$\nu_{A_i}$ in order to get $\nu_i\in {\rm ca}(\Sigma_{A_i},c_0(\aleph_{m_i}))$ such that 
$L_1(\nu_{A_i})=L_1(\nu_i)$ with equal norms. Let us consider the Banach space
$$
	Y:=\left(\bigoplus_{i=1}^n c_0(\aleph_{m_i})\right)_{\infty},
$$
which is isometric to~$c_0(\aleph_k)$. Finally, we can apply inductively Lemma~\ref{lem:glue}
to get $\tilde{\nu}\in {\rm ca}(\Sigma,Y)$ such that $L_1(\nu)=L_1(\tilde{\nu})$ with equivalent norms.
\end{proof}

Let $X$ be a Banach space, let $(\Omega,\Sigma)$ be a measurable space and let $\nu:\Sigma \to X$ be a map. 
The Orlicz-Pettis theorem (see, e.g., \cite[p.~22, Corollary~4]{die-uhl-J}) implies that $\nu \in {\rm ca}(\Sigma,X)$ if and only if 
the composition of $\nu$ with each $x^*\in X^*$ is countably additive. Diestel and Faires~\cite{die-fai} (cf. \cite[p.~23, Corollary~7]{die-uhl-J}) proved that
if $X$ contains no closed subspace isomorphic to~$\ell_\infty$ and $\Delta \sub X^*$ is a total set (i.e., $\bigcap_{x^*\in \Delta}\ker x^*=\{0\}$),
then $\nu \in {\rm ca}(\Sigma,X)$ if and only if the composition of $\nu$ with each $x^*\in \Delta$ is countably additive. As an application, we get part~(i)
of the following result, which also collects further properties of $c_0(\kappa)$-valued vector measures.

\begin{theo}\label{theo:remarks}
Let $(\Omega,\Sigma)$ be a measurable space and let $\nu:\Sigma \to c_0(\kappa)$ be a map, where $\kappa$ is a cardinal.
For each $\alpha<\kappa$, let $e_\alpha^*\in c_0(\kappa)^*$ be the $\alpha$th-coordinate projection and let
$\nu_\alpha: \Sigma \to \mathbb R$ be the composition of $\nu$ with $e_\alpha^*$. 
The following statements hold: 
\begin{enumerate}
\item[(i)] $\nu\in {\rm ca}(\Sigma,c_0(\kappa))$ if and only if $\nu_\alpha\in {\rm ca}(\Sigma,\mathbb R)$
for all $\alpha<\kappa$. 
\item[(ii)] If $\nu\in {\rm ca}(\Sigma,c_0(\kappa))$, then:
\begin{enumerate}
\item[(ii.a)] There is a countable set $\Gamma \sub \kappa$ such that $\bigcap_{\alpha\in \Gamma}\mathcal{N}(\nu_\alpha) \sub \mathcal{N}(\nu)$.
\item[(ii.b)] For each $\epsilon>0$ there is a countable partition
$\kappa=\bigcup_{n\in \N} \Gamma_{n,\epsilon}$ such that for every $n\in \N$ and for every $A\in \Sigma$ the set
$\{\alpha\in \Gamma_{n,\epsilon}: |\nu_\alpha(A)|>\epsilon\}$ has cardinality less than~$n$.
\end{enumerate}
\end{enumerate}
\end{theo}
\begin{proof} (i) This follows from the aforementioned result of Diestel and Faires applied to the total set $\Delta:=\{e_\alpha^*:\alpha<\kappa\} \sub c_0(\kappa)^*$
(bear in mind that weakly compactly generated Banach spaces, like $c_0(\kappa)$, contain no closed subspace isomorphic to~$\ell_\infty$).

(ii.a) Let $\mu$ be a Rybakov control measure of~$\nu$. Since $c_0(\kappa)^*=\ell_1(\kappa)$, there is $\varphi\in \ell_1(\kappa)$
such that $\mu=|\varphi\circ \nu|$. The set $\Gamma:=\{\alpha<\kappa:\varphi(\alpha)\neq 0\}$ is countable
and $(\varphi\circ \nu)(A)=\sum_{\alpha \in \Gamma}\varphi(\alpha)\nu_\alpha(A)$ for every $A\in \Sigma$, the series being
absolutely convergent. Clearly, the inclusion $\bigcap_{\alpha\in \Gamma}\mathcal{N}(\nu_\alpha) \sub \mathcal{N}(\mu)=\mathcal{N}(\nu)$ holds.

(ii.b) Let $K \sub c_0(\kappa)$ be the weak closure of the set $\{\nu(A):A\in \Sigma\}$.
Then $K$ is weakly compact (see, e.g., \cite[p.~14, Corollary~7]{die-uhl-J})  and, in fact, it is uniform Eberlein compact (i.e., it is homeomorphic to a weakly compact subset of a Hilbert space)
when equipped with the weak topology, \cite[Corollary~2.3]{rod15}. The conclusion follows from Farmaki's characterization of those compact subsets of sigma-products
which are uniform Eberlein compact, see \cite{far} (cf. \cite[Corollary~6.33(i)]{fab-alt-JJ}).
\end{proof}

\subsection*{Acknowledgements} The research was supported by grants PID2021-122126NB-C32 
(funded by MCIN/AEI/10.13039/501100011033 and ``ERDF A way of making Europe'', EU) and 
21955/PI/22 (funded by {\em Fundaci\'on S\'eneca - ACyT Regi\'{o}n de Murcia}).


\bibliographystyle{amsplain}

\end{document}